\documentclass[english,reqno]{amsart}

\usepackage{amsmath,amssymb,amsfonts,amsthm}
\usepackage[latin1]{inputenc}
\usepackage{color}
\usepackage{graphicx}
\usepackage{subfigure}

\newcommand{\supp}{\operatorname{supp}}
\newcommand{\im}{\operatorname{im}}
\newcommand{\lex}{\operatorname{lex}}
\newcommand{\Z}{\mathbb{Z}}
\newcommand{\N}{\mathbb{N}}
\newcommand{\R}{\mathbb{R}}
\renewcommand{\k}{\Bbbk}
\newcommand{\ttt}{\mathbf{t}}
\newcommand{\ini}{\operatorname{in}}

\renewcommand{\bar}{\overline}

\renewcommand{\aa}{\mathbf{a}}
\newcommand{\bb}{\mathbf{b}}
\newcommand{\xx}{\mathbf{x}}
\renewcommand{\tt}{\mathbf{t}}

\newtheorem{theorem}{Theorem}[section]
\newtheorem{cor}[theorem]{Corollary}
\newtheorem{prop}[theorem]{Proposition}
\newtheorem{lemma}[theorem]{Lemma}

\theoremstyle{definition}
\newtheorem{dfn}[theorem]{Definition}

\theoremstyle{remark}

\newtheorem{exa}[theorem]{Example}

\begin{document}

\title{The cones of Hilbert functions of squarefree modules}
\author[C. Bertone]{Cristina Bertone}
\address{%
Dipartimento di Matematica\\
Università di Torino}
\email{cristina.bertone@unito.it}

\author[D. H. Nguyen]{Dang Hop Nguyen}
\address{%
Insitut f\"ur Mathematik\\
Universit\"at Osnabr\"uck}
\email{nhop@uos.de}

\author[K. Vorwerk]{Kathrin Vorwerk}
\address{%
Institutionen för matematik\\
Kungliga Tekniska Högskolan}
\email{kathrinv@math.kth.se}

\thanks
{The first author was financially supported by the PRIN \lq\lq Geometria delle variet\'a algebriche e dei loro spazi di moduli\rq\rq, cofinanced by MIUR (Italy) (cofin 2008).
The second author is grateful to the support from the graduate school "Combinatorial structures in Algebra and Topology" at the University of Osnabr\"uck.
The third author was supported grant KAW 2005.0098 from by the Knut and Alice Wallenberg foundation.
}

\keywords{squarefree modules, Hilbert function, cones}
\subjclass[2010]{16W50, 13F55}

\begin{abstract}
In this paper, we study different generalizations of the notion of squarefreeness for ideals to the more general case of modules.
We describe the cones of Hilbert functions for squarefree modules in general and those generated in degree zero. We give their extremal rays and defining inequalities.
For squarefree modules generated in degree zero, we compare the defining inequalities of that cone with the classical Kruskal-Katona bound, also asymptotically.
\end{abstract}
\maketitle

\section{Introduction}

Squarefree monomial ideals and Stanley-Reisner rings have been intensively studied, because of their applications in many fields of combinatorics.
 It is quite natural to ask for a suitable generalization of the concept of squarefreeness to modules.

In Section \ref{sqfdef}, we focus on different possible definitions of squarefreeness for modules over the polynomial ring
$S=\k[x_1,\dots,x_n]$ with the standard $\N^n$-grading. While one of these definitions (cf. Definition \ref{sf1}) is in literature, the other ones are quite natural
 extension of properties of monomial squarefree ideals. We show that, eventually under some hypothesis on the degree of the generators of the
 module, these definitions turn out to be equivalent.

Recently, Boij and Söderberg \cite{BoijSoder} studied the cone of Betti diagrams of graded Cohen-Macaulay modules and conjectured that its extremal rays are given by Betti diagrams of pure resolutions which then was proved by Eisenbud and Schreyer \cite{EisenbudSchreyer}. This relates to the study of cones of Hilbert functions as it has been done for Artinian graded S-modules or modules of fixed dimension with a prescribed Hilbert polynomial \cite{BoijSmith}.

With those results as our motivation, we investigate the cone of Hilbert function of squarefree modules in Section \ref{cones}. We determine both the extremal rays and the defining inequalities of the cone of Hilbert functions of squarefree modules in Section \ref{conesqfd}.

Then, we restrict to the class of squarefree modules generated in degree zero in Section \ref{conegen0}. This case can be reduced to Hilbert functions of Stanley-Reisner rings using Gr\"obner bases. Again, we describe the extremal rays and defining inequalities of the cone of Hilbert functions of those modules.

The defining inequalities in this last case give a linear bound on the growth of the Hilbert function of a Stanley-Reisner ring. In Section \ref{comp}, we compare this bound to the non-linear but optimal bound given by the Kruskal-Katona Theorem. We compute the
maximal difference among the two bounds for a fixed number of variables $n$ and a fixed $d$-th entry of the $f$-vector.

Finally, in Section \ref{limit}, we study limits of those differences.

\section{Notation}

We start fixing some notations that we will use throughout the paper.

We write $[n] = \{1,\ldots,n\}$. A vector $\aa = (a_1,\dots , a_n)\in \N^n$ is called {\em squarefree} if $0 \leq a_i \leq 1$ for $i \in [n]$.
We set $|\aa| = a_1 +\cdots + a_n$. The support of $\aa$ is
$\supp(\aa) = \{i \ | \ a_i \neq 0 \} \subseteq [n]$.
Frequently, we will identify the squarefree vector $\aa$ and its support $F = \supp(\aa)$.

Let $\k$ be a field, $S = \k[x_1,\dots, x_n]$ is the symmetric algebra in $n$ indeterminates over $\k$. Also, $\mathfrak m =(x_1,\dots , x_n)$ is the graded maximal ideal of $S$. We denote by $\xx^\aa$ the monomial $x_1^{a_1}\cdots x_n^{a_n}$ with $\aa = (a_1,\dots , a_n)$.
The symmetric algebra $S$ has a natural $\N^n$-grading given by $\deg_{x_k} \xx^\aa = a_k$ for $k \in [n]$.

Denote by $\Lambda$ the standard graded exterior algebra in $n$ variables over $\k$. This is a graded associative algebra over $\k$. It is not commutative but skew-commutative in the sense that $ab=(-1)^{\deg a\deg b}ba$ for homogeneous elements $a,b \in \Lambda$ and $a^2=0$ if $a$ is homogeneous of odd degree. $\Lambda$ has the same natural $\N^n$-grading as $S$.

By a $\Lambda$-module $M$ we mean a finitely generated graded left $\Lambda$-module which is also a right $\Lambda$-module so that the actions of $\Lambda$ satisfy: $am = (-1)^{\deg a\deg m}ma$ for all homogeneous elements $a\in \Lambda, m\in M$.

For an element $u$ of an $\N^n$-graded vector space $M =\oplus_{\aa \in \N^n} M_\aa$, we write $\deg(u) = \aa$ if $u \in M_\aa$. We set $\supp(u) = \supp(\deg(u))$
and $|u| = |\deg(u)|$.

Consider a finitely generated  $\N^n$-graded module $M$ over $S$ or $\Lambda$. We denote its minimal free $\N^n$-graded resolution as
\[
0\longleftarrow M\stackrel{\phi_0}{\longleftarrow}F_0\stackrel{\phi_1}{\longleftarrow}F_1\stackrel{\phi_2}{\longleftarrow}\cdots \stackrel{\phi_r}{\longleftarrow}F_r\longleftarrow 0.
\]
Furthermore, let $A_i$ be the matrix of the map $\phi_i$.

Given an $\N^n$-graded module $M$ over $S$ or $\Lambda$, the $\N^n$-graded (or fine) {\em Hilbert function} of $M$ is given by
\[
    H_M(\aa) = \dim_\k M_\aa \qquad \text{for} \; \aa \in \N^n
\]
and its $\N^n$-graded (or fine) {\em Hilbert series} is
\[
    H(M,\tt) = \sum_{\aa \in \N^n} H_M(\aa) \tt^\aa
\]
as a power series in $\Z[[t_1,\ldots,t_n]]$.

Similarly, the $\N$-graded (or coarse) versions of the Hilbert function and the Hilbert series are
\[
   H_M(i) = \dim_\k M_i \quad \text{for} \quad i \in \N
   \qquad \text{and} \qquad
   H(M,t) = \sum_{i \in \N} H_M(i) t^i
\]
where $M_i = \bigoplus_{\aa \in \N^n, |\aa| = n} M_\aa$.

For general graded modules, it is natural to allow also negative degrees. However, this paper considers squarefree modules which makes sense only with all components in non-negative degrees.

\section{Squarefree $S$-modules}
\label{sqfdef}

The most common definition of a squarefree module in the literature is the following.

\begin{dfn}[Yanagawa, \cite{Ya}]\label{sf1}
A finitely generated $\N^n$-graded $S$-module $M = \oplus_{\aa \in \N^n} M_\aa$ is called {\em squarefree} if
the multiplication map $M_\aa \stackrel{x_i}{\rightarrow} M_{\aa + e_i}$ is a bijection for every $i \in \supp(\aa)$.
\end{dfn}

\begin{exa}
\label{exsq}
Canonical examples of squarefree $S$-modules arise from simplicial complexes. For a simplicial complex $\Delta$ on $n$ vertices, the Stanley-Reisner ideal $I_\Delta$ and the Stanley-Reisner ring $\k[\Delta] = S/I_{\Delta}$ are squarefree modules.

Also, a graded free module $S(-F)$ for $F\subseteq[n]$ is squarefree.
In particular, the $\Z^n$-graded canonical module of $S$,
$\omega_S=S(-\mathbf 1)$, is squarefree where $\mathbf 1=(1,\dots,1)$.
\end{exa}

%We will say that $\Delta$ is a simplicial complex {\it on $n$ vertices} and $\k[\Delta]$ is a Stanley-Reisner ring {\it on
%$n$ vertices}, if the vertex set of $\Delta$ is $[n]$ as in Example \ref{exsq}.

This definition of squarefreeness of an $S$-module $M$ turns out to be equivalent to certain properties of the minimal free resolution and the generators of $M$ which might be easier to check.

\begin{dfn}\label{sf2}
Let $M$ be an $\N^n$-graded finitely generated $S$-module with minimal resolution
\[
    0\longleftarrow M\stackrel{\phi_0}{\longleftarrow}F_0\stackrel{\phi_1}{\longleftarrow}F_1\stackrel{\phi_2}{\longleftarrow}\cdots \stackrel{\phi_r}{\longleftarrow}F_r\longleftarrow 0.
\]
We say that $M$ satisfies:
\begin{itemize}
\item \emph{condition $(F)$} if $F_i$ is generated in squarefree degrees for all $i = 0,\ldots,r$,
\item \emph{condition $(F_1)$} if $F_1$ is generated in squarefree degrees,
\item \emph{condition $(\phi)$} if the matrices $A_i$ corresponding to the maps $\phi_i$ have squarefree entries for all $i=0,\ldots,r$.
\end{itemize}
\end{dfn}

We will show that the various conditions in Definition \ref{sf2} are satisfied for all squarefree modules. Furthermore, each condition possibly together with an assumption on the degrees of the generators of $M$ implies squarefreeness.

\begin{prop}
A finitely generated  $\N^n$-graded $S$-module $M$ is squarefree if and only if it satisfies condition $(F)$.
\end{prop}

\begin{proof}
It is shown in (\cite[Corollary 2.4]{Ya}) that squarefree modules have squarefree $i$-th syzygies for all $i$. This shows that $(F)$ is satisfied for squarefree $M$.

Assume that $M$ satisfies condition $(F)$. As stated in (\cite[Lemma 2.3]{Ya}), cokernels of homogenous maps between squarefree modules are squarefree and thus generated in squarefree degrees. As indicated in Example \ref{exsq}, graded free modules are squarefree if and only if their shifts are $\{0,1\}$-vectors. This implies that $M$ is squarefree.
%If $M$ satisfies condition $(F)$, then $M$ is squarefree because it is the cokernel of a homogeneous map between squarefree modules.
%
%Vice versa, in \cite{Ya}, it is shown that kernels of a map between squarefree modules are squarefree too (\cite[Lemma 2.3]{Ya}) and so the module of $i$-th
% syzygies of $M$ is squarefree for all $i$ (\cite[Corollary 2.4]{Ya}).
\end{proof}

\begin{lemma}\label{sqfinduction}
Assume that in the minimal free resolution of $M$, the free module $F_{i-1}$ has squarefree generators and $A_{i}$ has squarefree entries.
 Then $F_i$ is generated in squarefree degrees.
\end{lemma}

\begin{proof}
Assume that $F_{i-1} = \oplus_j S e_j$ is generated in squarefree degrees $\deg e_j \in \{0,1\}^n$ and furthermore, that some homogeneous
generator $f$ of $F_{i}$ has non-squarefree degree $\deg f$. Then $\deg_{x_k} f \geq 2$ for some $k \in [n]$. We apply the differential map by
 multiplying with the squarefree matrix $A_i$ and get that
\[
    f \mapsto \sum_j a_j e_j
\]
where all $a_j$ are squarefree monomials. Because $\deg_{x_k} f = \deg_{x_k} a_j + \deg_{x_k} e_j$ for all $j$ and both $a_j$ and $e_j$ are in
 squarefree degrees, we find that $\deg_{x_k} a_j = 1$ for all $j$ where $a_j \not= 0$.
Thus, we can define
\[
    b_j = \begin{cases}
               a_j / x_k & a_j \not= 0 \\
               0         & a_j = 0
           \end{cases}
\]

Since we have a free resolution, $\sum_j a_j e_j$ belongs to $\ker A_{i-1}$. This implies that $x_k \sum_j b_j e_j$ belongs to  $\ker(A_{i-1})$ and because $F_{i-1}$ is free, also $\sum_j b_j e_j \in \ker(A_{i-1}) = \im(A_i)$. So we can write $\sum_j b_j e_j = A_i(g)$ for some $g \in F_{i}$. In particular, $f - x_k g \in \ker(A_i)$ and thus $f - x_k g \in \im(A_{i+1})$ which is a contradiction to the minimality of the resolution.
\end{proof}

\begin{prop}
A finitely generated $\mathbb N^n$-graded $S$-module $M$ which is generated in squarefree degrees satisfies $(F_1)$ if and only if  it satisfies condition $(F)$.
\end{prop}

\begin{proof}
Clearly, condition ($F$) implies condition ($F_1$) even without the additional assumption on the generators of $M$.

Vice versa, let $M$ be generated in squarefree degrees, then $F_0$ has squarefree generators. Since $F_1$ is squarefree by assumption,
the entries of $A_1$ are squarefree.
Again, $\ker(A_1)=\im(A_2)$ is kernel of a homogenous map between squarefree modules and thus generated in squarefree degree (\cite[Lemma 2.3]{Ya}). So the entries of $A_2$ must be squarefree.
To prove that $F_2$ has squarefree generators, we apply Lemma \ref{sqfinduction}.
Iterating these arguments, we find that $M$ satiesfies condition $(F)$.
\end{proof}

\begin{prop}
A finitely generated $\mathbb N^n$-graded $S$-module $M$ satisfying condition (F) also satisfies condition ($\phi$). The converse is true if $M$ is generated in squarefree degrees.
\end{prop}

\begin{proof}
If $M$ satisfies condition $(F)$, then it satisfies condition $(\phi)$ because the degrees of the entries of the $j$-th column of the matrix $A_i$ are componentwise bounded by the degree of the $j$-th generator of $F_{i+1}$.

Vice versa, let $(\phi$) be satisfied. We prove that $F_i$ is generated in squarefree degrees by induction on $i \geq 0$. Because $M$ is generated in squarefree degrees, then $F_0$ is generated in squarefree degrees. The inductive step is Lemma \ref{sqfinduction}.
\end{proof}

We summarize the equivalences among the conditions.

\begin{theorem}
\label{complequiv}
Given $M$ an finitely generated $\N^n$-graded $S$-module.
\begin{center}
    $M$ is squarefree
    $\Leftrightarrow$ $M$ satisfies condition ($F$) \\
    $\Rightarrow$ $M$ satisfies conditions ($F_1$) and ($\phi$).
\end{center}
If $M$ is generated in squarefree degrees, then this changes to:
\begin{center}
    $M$ is squarefree
    $\Leftrightarrow$ $M$ satisfies condition ($F$) \\
    $\Leftrightarrow$ $M$ satisfies conditions ($F_1$)
    $\Leftrightarrow$ $M$ satisfies condition ($\phi$).
\end{center}
\end{theorem}

\section{Cones of Hilbert functions of squarefree $S$- and $\Lambda$-modules}
\label{cones}

Consider the family of finitely generated squarefree $S$- or $\Lambda$-modules, or possibly a subfamily defined by some extra property.
The set of all (coarsely graded) Hilbert functions of modules in that family
forms a semigroup in the infinite-dimensional space of non-negative integer sequences $\N^\N$, that means it is closed under addition and multiplication with natural numbers. We will consider the cone that is spanned by this set in $\R^\N$ and call this {\em the cone of Hilbert functions of squarefree modules}. It is a finite-dimensional cone in $\R^\N$ which makes it possible for us to describe its defining inequalities and extremal rays.

Similarly, the set of Hilbert series of squarefree modules spans a finite-dimensional cone in $\R[[t]]$ which we call {\em the cone of Hilbert series of squarefree modules}.

The goal of this section is to show that the cones of Hilbert functions of squarefree $S$-modules and of $\Lambda$-modules are simplicial. We also describe their extremal rays and give their defining inequalities.

\subsection{Squarefree $S$-modules}
\label{conesqfd}

In this section, we describe the cone of Hilbert functions of squarefree modules $M$. We want to find a family of squarefree modules $M_\ell$ such that for any squarefree module $M$, it holds that
\[
    H_M(t)=\sum \alpha_\ell \ H_{M_\ell}(t)
\]
with $\alpha_\ell \geq 0$.

It turns out to be easier to work with the Hilbert series of $M$ as we will see below.
%First, we show an easy way to write the Hilbert series of a squarefree module $M$.

\begin{lemma}
\label{dimsqf}
If $M$ is squarefree, then $M_\aa \cong M_{\supp(\aa)}$ for all $\aa \in \N^n$. In particular, $\dim_\k M_\aa$ depends only on $\supp(\aa)$.
\end{lemma}
\begin{proof}
By definition there is a bijection between $M_\aa$ and $M_{\aa+e_i}$ for all $e_i \in \supp(\aa)$ and thus $M_\aa \cong M_{\aa+\bb}$ for all $\bb \in \N^n$ with $\supp(\bb) \subseteq \supp(\aa)$. But $\aa = \supp(\aa) + \bb$ where $\supp(\bb) \subseteq \supp(\aa)$, so $M_\aa \cong M_{\supp(\aa)}$ for all $\aa \in \N^n$ follows.
In particular, this implies that $\dim_\k M_\aa = \dim_\k M_{\supp(\aa)}$.
%Then if we consider $\sigma \subseteq[n]$,
%we have $\dim_{\k} M_\sigma=\dim_{\k} M_a$ for all $a\in \mathbb N^n$ such that $\supp(a)=\sigma$.
\end{proof}

\begin{prop}
\label{fineseries}
The fine graded Hilbert series of a squarefree module $M$ is given by
\[
    H(M,\ttt) = \sum_{\sigma \subseteq [n]} \dim_{\k} M_{\sigma}
        \prod_{i \in \sigma} \frac{t_i}{1-t_i}.
\]
\end{prop}

\begin{proof}
Using Lemma \ref{dimsqf}, we compute that
\begin{align*}
 H(M,\ttt)
     &= \sum_{\aa \in \N^n} \dim_\k M_\aa \ \ttt^\aa =\\
     &= \sum_{\sigma \subseteq [n]} \dim_\k M_\sigma
         \mathop{\sum_{\aa \in \N^n}}_{\supp(\aa)=\sigma} \ttt^\aa
     =\sum_{\sigma \subseteq [n]} \dim_\k M_\sigma  \prod_{i \in \sigma} \frac{t_i}{1-t_i}. \qedhere
\end{align*}
\end{proof}

\begin{cor}
\label{cor:coarseH}
The $\N$-graded Hilbert series of a squarefree $S$-module $M$ is given by
\[
    H(M,t)
        = \sum_{\sigma \subseteq [n]} \dim_\k M_{\sigma} \
            \frac{t^{|\sigma|}}{(1-t)^{|\sigma|}}.
\]
\end{cor}
 %For that, it suffices to describe the cone of Hilbert functions as the coordinate transformation with non-negative coefficients.

Looking at the proof of Proposition \ref{fineseries}, it is natural to consider modules generated in one squarefree degree only.

\begin{dfn}
\label{rays}
For any $0 \leq \ell \leq n$, define the squarefree module
\[
    N_\ell = \mathop{\bigoplus_{\aa \in \N^n}}_{\supp(\aa)=[\ell]} N_\aa,
\]
where $N_\aa \cong \k$ for all $\aa \in \N^n$ with $\supp(\aa)=[\ell]$.
\end{dfn}

Observe that the coarse graded Hilbert series of $N_\ell$ is
\begin{equation}
\label{raysH}
    H(N_\ell,t) = {t^\ell}/{(1-t)^\ell}.
\end{equation}

\begin{theorem}
\label{series}
For any squarefree module $M$, we get
\begin{equation}
\label{combinextrrays}
    H(M,t) = \sum_{\sigma \subseteq [n]} \dim_\k M_\sigma \ H(N_{|\sigma|}, t).
\end{equation}
In particular, the cone of Hilbert series of squarefree modules is simplicial and its extremal rays are the Hilbert series $H(N_\ell,t)$ for $0 \leq \ell \leq n$.
\end{theorem}

\begin{proof}
Equation \eqref{combinextrrays} follows directly from Corollary \ref{cor:coarseH} and Equation \eqref{raysH}. We observe that it can be written as
\[
    H(M,t) = \sum_{0 \leq \ell \leq n} \alpha_\ell \ H(N_\ell, t).
\]
where $\alpha_\ell = \sum_{\sigma \subseteq [n], |\sigma| = \ell} \dim_\k M_\sigma \geq 0$.
We check that the Hilbert series $H(N_\ell,t)$ for $0 \leq \ell \leq n$ are linearly independent.
\end{proof}

%\begin{cor}
%The cone of Hilbert functions of squarefree modules is simplicial and its extremal rays are the Hilbert functions $H(M_l,t)$ for $0 \leq l \leq n$.
%\end{cor}

\begin{cor}
\label{firstfacets}
The cone of Hilbert functions of squarefree modules $M$ has the following $n+1$ defining inequalitites
\begin{equation}
\label{defineq}
    H_M(i) \geq \sum_{j=0}^{i-1} (-1)^{i+j-1} \binom{i}{j} H_M(j),
\end{equation}
where $i=0,\ldots,n$. %(Convention: for $i=0$, the right-hand side is zero.)
\end{cor}

\begin{proof}
We consider the linear system of $n+1$ equations
\[
    H_M(i) = \sum_{j=0}^i \alpha_j \binom{i}{j}, \quad 0 \leq i \leq n,
\]
where
\[
\alpha_j=\sum_{\sigma\subseteq [n], |\sigma|=j} \dim_\k M_\sigma
\]
is the coefficient of $H(N_j,t)$ in Equation \eqref{combinextrrays}.
We invert the $(n+1) \times (n+1)$-matrix whose entries are $\binom{i}{j}$ for $0 \leq i,j ,\leq n$ and get
\[
    \alpha_i = \sum_{j=0}^i (-1)^{i+j}\binom{i}{j} H_M(j), \quad 0\leq i\leq n.
\]
We use that $\alpha_{i} \geq 0$ and we conclude that
\[
    H_M(i) \geq \sum_{j=0}^{i-1} (-1)^{i+j}\binom{i}{j} H_M(j). \qedhere
\]
\end{proof}

\begin{exa}
Consider the monomial squarefree ideal $I=(xy,xzt,yt)$ in the polynomial ring $\k[x,y,z,t]$. As shown in Theorem \ref{series}, we can write the Hilbert series of $I$ as
\begin{align*}
    H(I,t)
    &= \sum_{\sigma \subseteq [n]} \dim_\k I_\sigma \ H(N_{|\sigma|}, t) \\
    &= (\dim_\k I_{xy}+\dim_\k I_{yt})\frac{t^2}{(1-t)^2} \\
        & \qquad +(\dim_\k I_{xyz}+\dim_\k I_{xyt} + \dim_\k I_{xzt}
        + \dim_\k I_{yzt})\frac{t^3}{(1-t)^3} \\
        & \qquad + \dim_\k I_{xyzt}\frac{t^4}{(1-t)^4} \\
    &= 2\frac{t^2}{(1-t)^2}+4\frac{t^3}{(1-t)^3}+\frac{t^4}{(1-t)^4}.
\end{align*}

It is easy to check that for every $j=0,\dots,4$, the Hilbert function $H_I(j)$ satisfies Inequality \eqref{defineq} of Corollary \ref{firstfacets}.
\end{exa}

\subsection{Squarefree $S$-modules generated in degree zero}
\label{conegen0}

In this section, we restrict our attention to squarefree modules generated in degree zero. It turns out that their Hilbert functions are closely related to Hilbert functions of Stanley-Reisner rings.
%They will turn out to be the more natural generalization of the concept of \lq\lq squarefreeness\rq\rq\ for ideals. KV: WHY?

First, we recall some of the theory of initial ideals for $S$-modules. For that, let $M$ be a quotient of a free $\N^n$-graded $S$-module with an $\N^n$-graded submodule $N$ whose generators are all in squarefree degrees:
\[
    M = S^k / N.
\]
%We will show that the cone of Hilbert functions of those modules is equal to the cone of Hilbert functions of Stanley-Reisner rings.

Write
\[
    S^k = S e_1 \oplus \ldots \oplus S e_k
\]
where $\deg e_i = 0$ for all $i \in [k]$.

\begin{dfn}
\label{lex}
The lexicographic monomial order on monomials of $S^k$ is defined by
\[
    \xx^\aa e_i < \xx^\bb e_j, \qquad \mbox{if } j<i \mbox{ or } j=i \mbox{ and } \xx^\aa <_{\lex} \xx^\bb
\]
where $<_{\lex}$ denotes the usual lexicographical order on monomials of $S$.
\end{dfn}

As usual, we can define the initial form of an element of a graded submodule $N$ of $S^k$ and the initial module $\ini(N)$ of $N$.
For details, we refer to Eisenbud \cite[Chapter 15]{Eis}.

\begin{prop}[{\cite[Theorem 15.26]{Eis}}]
\label{eisin}
Given an $\N^n$-graded submodule $N$ of $S^k$, then $M = S^k/N$ and $M' = S^k/\ini(N)$ have the same Hilbert function.
\end{prop}

Proposition \ref{eisin} allows us to consider the initial module $\ini(N)$ instead of a submodule $N$. Such initial modules have a very special form if $N$ is generated in squarefree degrees.
% with homogeneous and squarefree generators, its initial module.

\begin{prop}
Given an $\N^n$-graded submodule $N$ of $S^k$ that is generated in squarefree degrees, then $\ini(N)$ with respect to the term order of Definition \ref{lex} is an $\N^n$-graded submodule of $S^k$ of the form $I_1 \oplus \ldots \oplus I_k$ where each ideal $I_j$ is monomial and generated in squarefree degrees.
\end{prop}

\begin{proof}
The lexicographic order of Definition \ref{lex} only allows monomial terms of the form $m e_i$, where $x$ is a monomial in $S$, as initial terms. Thus, $\ini(N)$ is of the form
\[
    I_1 \oplus \cdots \oplus I_k
\]
where $I_1,\ldots,I_k$ are monomial ideals. The generators of each ideal $I_j$ are squarefree because each element added to the set of generators during Buchberger's algorithm \cite[Algorithm 15.9]{Eis} is homogeneous and squarefree.
\end{proof}

\begin{exa}
We consider $S=\k[x,y,z]$ with the fine grading and the module $S^3/M$, where $M$ is generated by the homogeneous elements $g_1=(xy,-xy,0),g_2=(2yz,0,2yz), g_3=(0,xyz,xyz),g_4=(2xz,-xz,xz)$. Using the term order of Definition \ref{lex}, we can compute the reduced Gr\"obner basis of $M$, obtaining:
\[
    \{g_1,\ g_2,\ g_4,\ g_5,\ g_6\},
    \quad \text{with} \ g_5 = (0,xyz,0),\ g_6=(0,0,xyz).
\]
The initial ideal of $M$ is generated by
\[(xy,0,0),\ (yz,0,0),\ (xz,0,0),\ (0,xyz,0),\ (0,0,xyz).\]
\end{exa}

\begin{cor}
The cone of Hilbert functions of squarefree $S$-modules that are generated in degree zero is equal to the cone of Hilbert functions of Stanley-Reisner rings over $S$.
\end{cor}

This motivates to study the cone of Hilbert functions Stanley-Reisner rings. We find that its extremal rays are Hilbert functions of modules similar to those chosen in Definition \ref{rays}.

\begin{dfn}
For any $0 \leq \ell \leq n$, define the simplicial complex
\[
    \Delta_\ell = \{ \sigma \subseteq [n] : |\sigma| \leq \ell \}
\]
which is the $(\ell-1)$-dimensional skeleton of the full simplex on vertex set $[n]$.
%Let $M_l = \k[\Delta_l]$ be its Stanley-Reisner ring.
\end{dfn}

Using \cite[Theorem 1.4]{Stan}, we compute the $\N$-graded Hilbert series of $\k[\Delta_\ell]$ as
\[
    H(\k[\Delta_\ell],t)
        = \sum_{i=0}^\ell f_{i-1}(\Delta_\ell) \frac{t^i}{(1-t)^i}
        = \sum_{i=0}^\ell \binom{n}{i} \frac{t^i}{(1-t)^i}
\]
where $f_{i-1}(\Delta_\ell) = \binom{n}{i}$ is the number of $(i-1)$-dimensional faces of $\Delta_\ell$.

\begin{prop}
\label{SRseries}
For any simplicial complex $\Delta$ on $n$ vertices, the Hilbert series $H(\k[\Delta],t)$ can be written as
\begin{equation}\label{rays0}
    H(\k[\Delta],t) = \sum_{\ell=0}^n \alpha_\ell \ H(\k[\Delta_\ell],t)
\end{equation}
where
\begin{equation}\label{alpha0}
    \alpha_\ell = \frac{f_{\ell-1}}{\binom{n}{\ell}} - \frac{f_\ell}{\binom{n}{\ell+1}},
\end{equation}
with the convention that $\frac{f_n}{\binom{n}{n+1}} = 0$.
\end{prop}
\begin{proof}
If $(f_{-1},\dots, f_{n-1})$ is the $f$-vector of $\Delta$, then its Hilbert series is
\[
    H(\k[\Delta],t) = \sum_{i=0}^n f_{i-1} \frac{t^i}{(1-t)^i}.
\]
In order to satisfy Equation \eqref{rays0}, the numbers $\alpha_\ell$ have to solve the following system of linear equations
\begin{equation}
\label{falpha}
    f_{i-1}= \binom{n}{i} \sum_{\ell=i}^n \alpha_\ell, \quad i=0,\ldots,n.
\end{equation}
The solutions of this system are exactly
\[
     \alpha_\ell
         = \frac{f_{\ell-1}}{\binom{n}{\ell}}
         - \frac{f_\ell}{\binom{n}{\ell+1}}, \quad j=0,\dots,n. \qedhere
\]
\end{proof}
%\begin{cor}
%The Hilbert series $H(M_l,t)$ for $l=0,\ldots,n$ form the extremal rays of the cone of Hilbert functions of Stanley-Reisner rings.
%\end{cor}

\begin{cor}
\label{cor:doublecount}
The Hilbert functions $H_{\k[\Delta_\ell]}$ for $\ell=0,\ldots,n$ form the extremal rays of the cone of Hilbert functions of Stanley-Reisner rings
over $S$.
\end{cor}

\begin{proof}
Observe that the condition of the numbers $\alpha_i$ as defined in \eqref{alpha0} to be non-negative, is equivalent to the inequality
\begin{equation*}
 \frac{(n-i)f_{i-1}}{i+1}\geq f_{i}.
\end{equation*}
We claim that this inequality always holds for a simplicial complex.
This can be seen by a double-counting argument: Indeed, each $(i-1)$-dimensional face of $\Delta$ is contained in at most $(n-i)$
faces of dimension $i$, so the left-hand side bounds above the number of $i$-dimensional faces.

We also see that the Hilbert series $H(\k[\Delta_\ell],t)$ for $\ell=0,\dots,n$, are linearly independent.
\end{proof}

\begin{cor}
\label{secondfacets}
The defining inequalities of the cone of Hilbert functions of Stanley-Reisner rings over $S$ vertices are given by
\begin{equation}
\label{defineqSR}
    H_M(k+1) \leq \sum_{i=1}^k (-1)^{i+k} \binom{k-1}{i-1} \left( \frac{k(k+1)}{k-i+1} + {n-k} \right) H_M(i)
\end{equation}
\end{cor}

\begin{proof}
Using \cite[Theorem 1.4]{Stan}, we have for every $\Delta$
\[
    H_\Delta(i)=\sum_{j=0}^{i-1} f_j\binom{i-1}{j}
\]
Considering $i=0,\dots,n$, we get the inverse equalities
\[
    f_j = \sum_{i=1}^{j+1}(-1)^{(i+j+1)}\binom{j}{i-1} H_\Delta(i)
\]
We substitute this expression in inequality \eqref{alpha0} and obtain the result.
\end{proof}

\begin{exa}
Consider the polynomial ring $S=\k[x,y,z,t]$ and the Stanley Reisner ring $\k[\Delta]=S/I_\Delta$, with $I_\Delta=(xy,xzt,yt)$. The $f$-vector of $\Delta$ is in this case $(1,4,4,0,0)$.
Using Proposition \ref{SRseries}, we write the Hilbert series of $\k[\Delta]$ as a combination of the Hilbert series of $\k[\Delta_\ell]$ for $\ell=0,\dots,3$.
\[
     H(\k[\Delta],t) = \sum_{i=0}^3 \alpha_\ell H(\k[\Delta_\ell],t)
         = \frac{1}{2} H(\k[\Delta_1],t) + \frac{1}{2} H(\k[\Delta_2],t).
\]
We see that the inequalities \eqref{defineqSR} of Corollary \ref{secondfacets} are satisfied.
\end{exa}
%\begin{remark}
%As already pointed out, (D0) is a particular case of (DSQF), so the bound  inequality of Corollary \ref{firstfacets} holds also in this case.%, which is not contradictory with Corollary \ref{secondfacets}.
%\end{remark}

\subsection{$\Lambda$-modules generated in degree zero}
\label{sec:hlambda0}

We can generalize the result of Proposition \ref{SRseries} and Corollaries \ref{cor:doublecount} and  \ref{secondfacets} to the more general setting of $\Lambda$-modules.

Let $M$ be a $\Lambda$-module that is finitely generated in degree zero. In a similar way as in Section \ref{conegen0}, we find that the Hilbert function of $M$ is equal to the Hilbert function of a $\Lambda$-module $M'$ that is generated in degree zero and has the form
\[
    M' = \bigoplus_{j=1}^k \Lambda / I_j
\]
where all $I_j$ are monomial ideals in $\Lambda$.

However, each $\Lambda$-module $M$ of the form $M = \Lambda / I$, where $I$ is a monomial ideal, can be identified with a simplicial complex $\Delta$ such that
\[
    H_M(i) = \dim_k \left( \Lambda / I_j \right)_i = f_{i-1}
\]
where $(f_{-1},f_0,\ldots,f_{n-1})$ is the $f$-vector of $\Delta$. Conversely, for each simplicial complex $\Delta$, we can define a $\Lambda$-module $M$ that is finitely generated in degree zero and that satisfies $H_M(i) = f_{i-1}$ for each $0 \leq i < n$.
As we already saw in the proof of Corollary \ref{cor:doublecount}, this implies the following corollary.

\begin{cor}
The cone of Hilbert functions of $\Lambda$-modules that are finitely generated in degree zero is simplicial and its defining inequalities are given by
\begin{equation}
\label{eq:hlambda0}
    \frac{H_M(i+1)}{\binom{n}{i+1}} \leq \frac{H_M(i)}{\binom{n}{i}}
\end{equation}
for $0 \leq i < n$. \qed
\end{cor}

\section{Comparison between linear and non-linear bounds}
\label{comp}

In Section \ref{conegen0} we found the defining linear inequalities \eqref{eq:hlambda0} of the cone of Hilbert functions of a $\Lambda$-modules that are finitely generated in degree zero. This is true because the Hilbert functions $H_M(\cdot)$ are basically identical to sums of $f$-vectors of simplicial complexes.

Throughout this section, we will write
\[
    \lambda_d = H_M(d)
\]
for a given $\Lambda$-module $M$ and for $d \geq 0$.

However, for simplicial complexes and thus $\Lambda$-modules $M$ generated in degree zero, there are the (non-linear) Kruskal-Katona inequalities.\\
Given two positive integers $\lambda$ and $d$, there is a unique way to expand $\lambda$ as a sum of binomial coefficients
\[
    \lambda = \binom{k_d}{d} + \binom{k_{d-1}}{d-1} + \ldots + \binom{k_2}{2} + \binom{k_1}{1}
\]
where $k_d > k_{d-1} > \ldots > k_2 > k_1 \geq 0$.
We define
\[
    \lambda^{[d]} =\binom{k_d}{d+1}+\binom{k_{d-1}}{d}+\ldots+\binom{k_j}{j+1}.
\]

In terms of $\lambda_d$, the Kruskal-Katona inequalities state that
\[
    \lambda_{d+1} \leq \lambda_d^{[d]}
\]
for $0 \leq d < n$.

For every given $\lambda_d \leq \binom{n}{d}$ there is a lex-segment ideal $I$ in $\Lambda$ which satisfies the Kruskal-Katona bound with equality. Thus, the linear bound will always be larger or equal to the Kruskal-Katona bound.

In this section, we investigate for which $\lambda_d$ the non-negative difference
\begin{equation}
\label{KKin}
    \frac{1}{\binom{n}{d}} \lambda_d - \frac{1}{\binom{n}{d+1}} \lambda_d^{[d]}
\end{equation}
gets maximal.

\begin{figure}[ht]
\begin{center}
\includegraphics[scale=0.5]{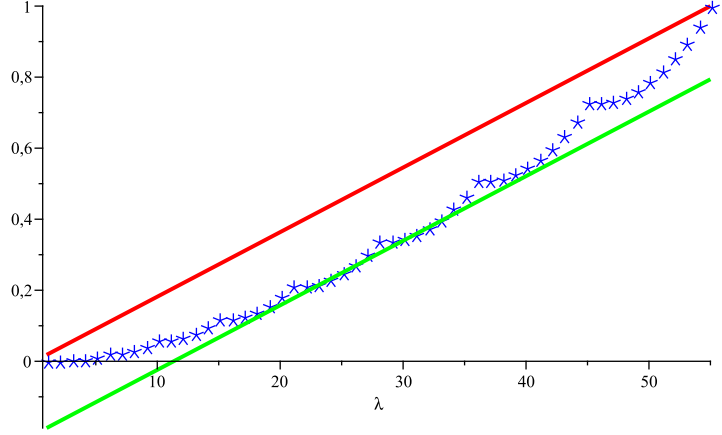}
\caption{Comparison between the Kruskal-Katona bound and the linear bound for $n=11$, $d=2$}\label{fig:1}
\end{center}
\end{figure}

\begin{dfn}
For $0 \leq \lambda \leq \binom{n}{d}$, define
\[
    \delta_{n,d}(\lambda)
      = \frac{1}{\binom{n}{d}} \lambda - \frac{1}{\binom{n}{d+1}} \lambda^{[d]}
\]
to be the difference between the linear bound and the Kruskal-Katona bound for $\lambda$ and define
\[
    \bar{\delta}_{n,d}
      = \max_{0 \leq \lambda \leq \binom{n}{d}} \delta_{n,d}(\lambda)
\]
to be the maximal difference for fixed $n$ and $d$.
\end{dfn}

We assume $n$ to be fixed throughout this section and write $\delta_d$ and $\bar{\delta}_d$ instead of $\delta_{n,d}$ respectively $\bar{\delta}_{n,d}$. In the next section, we will vary $n$ and use the notation $\delta_{n,d}$ instead.

We will compute for which $\lambda$ this maximal difference is achieved. As should be expected, the nature of the function $\lambda^{[d]}$ plays an important role.

\begin{lemma}\label{lemmakk}
Fix some $k_d, \ldots, k_1$ with $k_d > \ldots > k_1 \geq 0$ and fix some $i \in [d]$. Define
\[
    \lambda(k) = \binom{k_d}{d} + \ldots + \binom{k_{i+1}}{i+1} + \binom{k}{i} + \binom{k_{i-1}}{i-1} + \ldots + \binom{k_1}{1}
\]
Then the maximum value of $\delta_d(\lambda(k))$ depending on $k$ is achieved if
\[
    k = \left\lfloor \frac{i(n+1)}{d+1} \right\rfloor
    \quad \text{or} \quad
    k = \left\lceil \frac{i(n+1)}{d+1} - 1 \right\rceil.
\]
\end{lemma}

\begin{proof}
Certainly, if $\delta_d(\lambda(k))$ is maximal among all $k$, then $\delta_d(\lambda(k)) \geq \delta_d(\lambda(k-1))$ and $\delta_d(\lambda(k)) \geq \delta_d(\lambda(k+1)$. We investigate for which $k$ this is satisfied.
Assume $\delta_d(\lambda(k)) \geq \delta_d(\lambda(k-1))$. We compute
\begin{eqnarray*}
    0 &\leq& \delta_d(\lambda(k)) - \delta_d(\lambda(k-1))
        \\&=& \frac{1}{\binom{n}{d}} \lambda(k) -
            \frac{1}{\binom{n}{d+1}} \lambda(k)^{[d]} -
            \left(\frac{1}{\binom{n}{d}} \lambda(k-1) -
            \frac{1}{\binom{n}{d+1}} \lambda(k-1)^{[d]} \right)
        \\&=& \frac{1}{\binom{n}{d}} \left( \lambda(k) - \lambda(k-1) \right)
            - \frac{1}{\binom{n}{d+1}} \left( \lambda(k)^{[d]} -
            \lambda(k-1)^{[d]} \right)
        \\&=& \frac{1}{\binom{n}{d}} \left( \binom{k}{i}
            - \binom{k-1}{i} \right)
            - \frac{1}{\binom{n}{d+1}}\left( \binom{k}{i+1}
            - \binom{k-1}{i+1} \right)
        \\&=& \frac{1}{\binom{n}{d}} \binom{k-1}{i-1}
            - \frac{1}{\binom{n}{d+1}}\binom{k-1}{i}
\end{eqnarray*}
This implies that
\[
    \frac{k-i}{i}
        = \frac{ \binom{k-1}{i} }{ \binom{k-1}{i-1} }
        \leq \frac{\binom{n}{d+1}}{\binom{n}{d}} = \frac{n-d}{d+1}
\]
which can be reformulated as $k \leq \frac{i(n+1)}{d+1}$.
In a similar way we find that
$\delta_d(\lambda(k)) \geq \delta_d(\lambda(k+1))$
is satisfied only if $k \geq \frac{i(n+1)}{d+1} - 1$.
This implies that $\delta_d(\lambda(k))$ is maximal only if
$k = \left\lfloor \frac{i(n+1)}{d+1} \right\rfloor$ or
$k = \left\lceil \frac{i(n+1)}{d+1} - 1 \right\rceil$.
Both numbers are the same unless $\frac{i(n+1)}{d+1}$ is an integer. In that case, we get two consecutive numbers $k$ for which $\delta_d(\lambda(k))$ has the same value.

Because $\delta_{d}(\lambda_d(k))$ has to be maximal for some $k$, we find that $\delta_{d}(\lambda_d(k))$ is maximal for exactly those $k$ as above.
\end{proof}

In view of the previous lemma, we define
\[
    \bar{k}_i = \left\lfloor \frac{i(n+1)}{d+1} \right\rfloor.
\]

\begin{prop}
\label{maximum}
The maximal difference between the linear bound and the Kruskal-Katona bound is obtained for
\[
    \bar{\lambda}_d = \binom{\bar{k}_d}{d} + \ldots + \binom{\bar{k}_1}{1}.
\]
\end{prop}

\begin{proof}
We will show that for any $\lambda$ it holds that $\delta_d(\lambda) \leq \delta_d(\bar{\lambda}_d)$.

Let
\[
    \lambda = \binom{k_d}{d} + \ldots + \binom{k_1}{1}
\]
with $k_d > \ldots > k_1 \geq 0$ be the $d$-binomial expansion of $\lambda$ and assume that $\delta_d(\lambda) < \delta_d(\bar{\lambda}_d)$.

Then we find some $i \in [d]$ such that $k_i \not= \bar{k}_i$. Define
\[
    \lambda' = \binom{k_d}{d} + \ldots + \binom{k_{i+1}}{i+1} + \binom{\bar{k}_i}{i} + \binom{k_{i-1}}{i-1} + \ldots + \binom{k_1}{1}
\]
By the previous lemma, we have that $\delta_d(\lambda') \geq \delta_d(\lambda)$.

Repeatedly apply this step until
$\delta_d(\lambda) = \delta_d(\bar{\lambda}_d)$. If $k_i \not= \bar{k}_i$ for some $i$, then by the previous lemma it must hold that $k_i = \bar{k}_i - 1$ and that $\frac{i(n+1)}{d+1}$ is an integer. Then, we can replace $k_i$ by $\bar{k}_i$ in the $d$-binomial expansion for $\lambda$ without changing $\delta_d(\lambda)$.
\end{proof}

\begin{exa}
If we consider $n=11$ and $d=2$, we get that the maximal value of $\delta_{n,d}(k)$ is obtained for $\lambda=\binom{k_2}{2}+\binom{k_1}{1}$ with $k_i=i\cdot \frac{12}{3}=4i$ or $k_i=i\cdot \frac{12}{3}-1=4i-1$, for $i=1,2$.

Then the maximal value of $\delta_{n,d}(\lambda)$ is obtained for $\lambda\in\{24,25,31,32\}$, as shown in Figure \ref{fig:2}.

\begin{figure}
\centering
\subfigure[Range $\lambda=22,\dots,34$]{\includegraphics[scale=0.3]{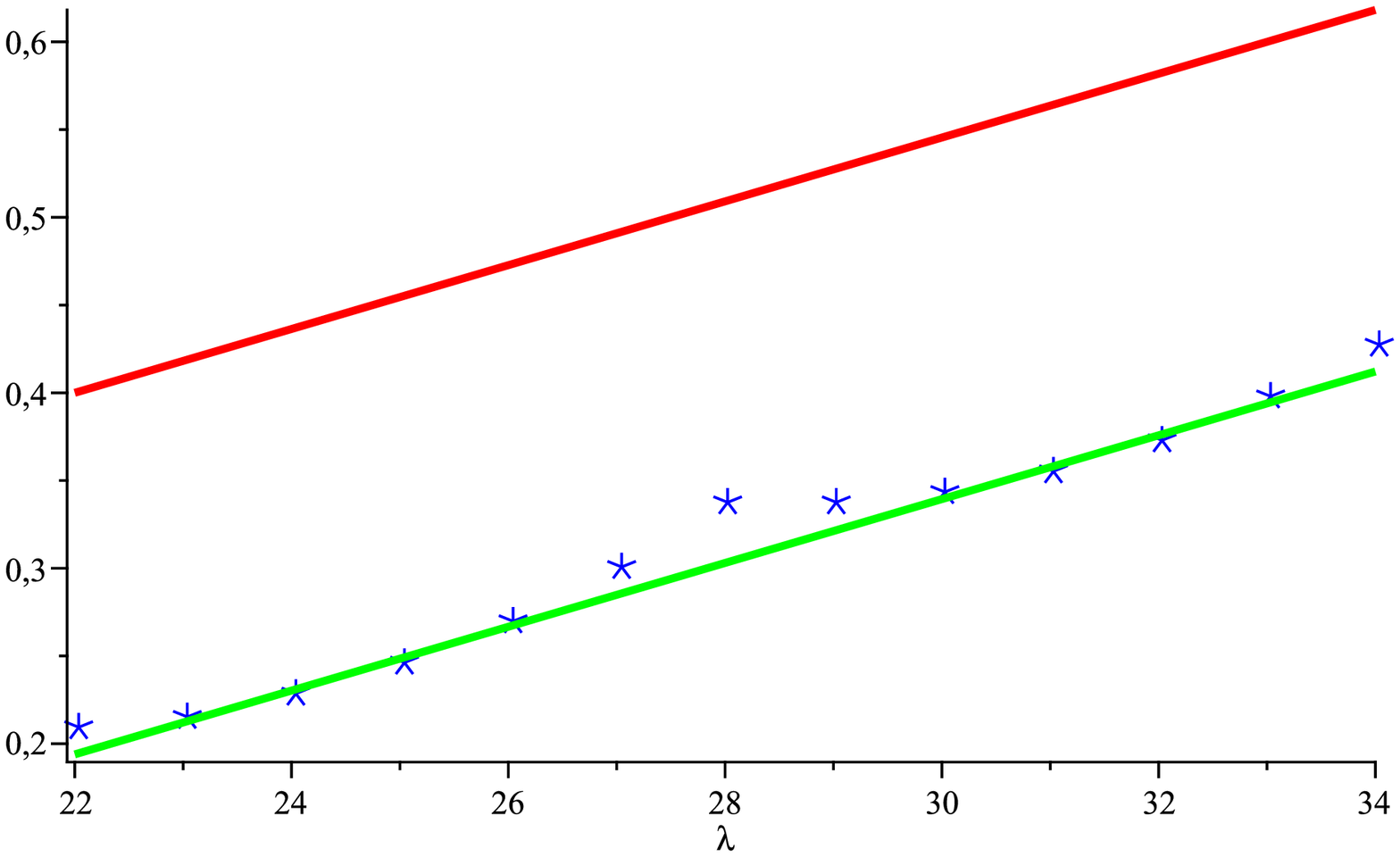}}\quad
\subfigure[Only the points maximizing $\delta_{n,d}$]{\includegraphics[scale=0.3]{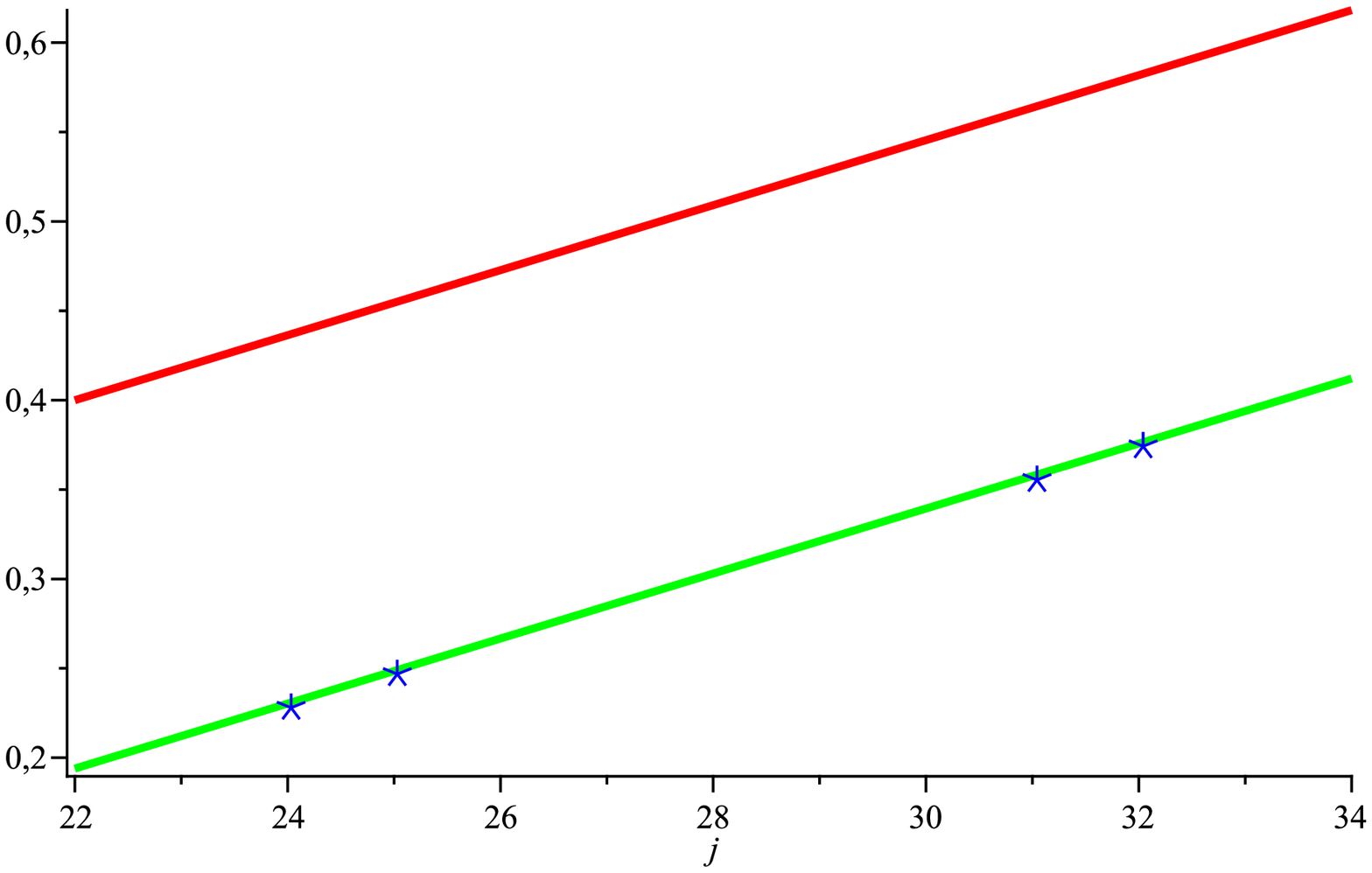}}
\caption{Zoom of Figure \ref{fig:1}}\label{fig:2}
\end{figure}
\end{exa}

\section{Limit of maximal differences between linear and non-linear bounds}
\label{limit}

We keep the notations of the previous section. In this section, we investigate the limits $\lim \limits_{n \rightarrow \infty} \delta_{n,d}$ and
$\lim \limits_{n \rightarrow \infty} \delta_{n,n-t}$ for fixed $d$ and $t$. The results will illustrate the asymptotic behavior of the difference between linear bounds and Kruskal-Katona bounds on Hilbert functions of $\Lambda$-modules that are generated in degree zero.

For $d=1$, the result follows directly from a short computation.

\begin{prop}
The maximal difference $\delta_{n,1}$ is given by
\[
    \delta_{2m,1} = \frac{m}{2(2m-1)}
    \quad \text{and} \quad
    \delta_{2m+1,1}=\frac{m+1}{2(2m+1)}
\] for all $m \geq 1$. \qed
\end{prop}

For $d \geq 2$, some more serious computations are necessary to get a result.

\begin{lemma}
\label{lem:dineq}
For $d\ge 2$, the following inequalities hold
\begin{equation}
\label{est1}
    \delta_{n,d} \geq \frac{d^d}{(d+1)^{d+1}}{\frac{(n-d)^{d+1}}{(n-d)(n-d+1)\cdots n}}
\end{equation}
and
\begin{equation}
\label{est2}
    \delta_{n,d} \leq
        \frac{d^d}{(d+1)^{d+1}}{\frac{(n-\frac{d-1}{2})^{d+1}}{(n-d)(n-d+1)\cdots n}}+ \frac{(d+1)(n+1)} {(d-1)(n-d)^2}.
\end{equation}
\end{lemma}

\begin{proof}
For legibility, we write $k$ for $k_d$ and $k_i$ for $\bar{k}_i$ for $i=1,\ldots d-1$. We compute that
\begin{eqnarray}
\label{ii}
    \delta_{n,d}
        &=& \frac{\bar{\lambda}_d}{\binom{n}{d}} -
        \frac{\bar{\lambda}_d^{[d]}}{\binom{n}{d+1}} \nonumber\\
        &=& \frac{(n-k)(k-d+1)(k-d+2)\cdots k}{(n-d)(n-d+1)\cdots n}\nonumber \\
        &&\qquad +\frac{1}{\binom{n}{d+1}} \sum_{i=1}^{d-1} \frac{\binom{k_i}{i}}{(i+1)}
        \left( \frac{(i+1)(n+1)}{d+1}-k_i-1 \right).
\end{eqnarray}

We prove the lower bound first.
Because $k_i \le \frac{i(n+1)}{d+1}$, we have that $$\frac{(i+1)(n+1)}{d+1}-k_i-1\ge 0.$$
Hence
\[
    \delta_{n,d}
        \geq \frac{(n-k)(k-d+1)(k-d+2)\cdots k}{(n-d)(n-d+1)\cdots n}
        \geq \frac{(n-k) (k-d+1)^d}{(n-d)^d}
\]
We see that
$k - d + 1 = \left\lfloor \frac{d(n+1)}{d+1} \right\rfloor  -d+1 \ge \frac{d(n+1)}{d+1}-d = \frac{d(n-d)}{d+1}$
and
$n-k \ge n - \frac{d(n+1)}{d+1} = \frac{n-d}{d+1}$.
Combining these inequalities yields the lower bound.

Next, we prove the upper bound.
We invoke the inequality between arithmetic mean and geometric mean for the $(d+1)$ numbers $k-d+1, k-d+2, \ldots, k$ and $d(n-k)$ and get

\begin{align}
\label{iii}
    (n-k)(k-d+1)(k-d+2)\cdots k
        &\le {\frac{1}{d}\left(\frac{dn-d(d-1)/2}{d+1}\right)^{d+1}}\nonumber \\
        &= \frac{d^d}{(d+1)^{d+1}}{\left(n-\frac{d-1}{2}\right)^{d+1}}.
\end{align}
For each $i=1,\ldots, d-1$, we have
\begin{equation}
\label{iv}
\frac{(i+1)(n+1)}{d+1}-k_i-1 \le \frac{(i+1)(n+1)}{d+1} - \frac{i(n+1)}{d+1} = \frac{n+1}{d+1}.
\end{equation}
On the other hand,
$$k_i\le k_{d-1}-(d-1)+i \leq \frac{(d-1)(n-d)}{d+1} + i$$

for all $i=1,\ldots, d-1$ and also $1+(d-1)\frac{n+1}{d+1} \le n$. Thus
\begin{align}
\label{v}
    \sum_{i=1}^{d-1} \frac{\binom{k_i}{i}}{i+1}
        &\le \sum_{i=1}^{d-1} \frac{\binom{\frac{(d-1)(n+1)}{d+1}-(d-1)+i}{i}}{i+1}\nonumber \\
        &= \frac{(d+1)\binom{1+(d-1)\frac{n+1}{d+1}}{d}-((d-1)(n-d)+d+1)}{(d-1)(n-d)}\nonumber \\
        &\le \frac{(d+1)\binom{1+(d-1)\frac{n+1}{d+1}}{d}}{(d-1)(n-d)} \nonumber\\
        &\leq \frac{(d+1)\binom{n}{d}}{(d-1)(n-d)}.
\end{align}
From \eqref{ii}, \eqref{iii}, \eqref{iv} and \eqref{v} we get
$$
\delta_{n,d} \le \frac{d^d}{(d+1)^{d+1}}{\frac{(n-\frac{d-1}{2})^{d+1}}{(n-d)(n-d+1)\cdots n}}+
\frac{1}{\binom{n}{d+1}}\frac{(d+1)\binom{n}{d}}{(d-1)(n-d)} \frac{n+1}{d+1}
$$
which is exactly the upper bound.
\end{proof}

\begin{prop}
For all $d \ge 1$ it holds that
\[
    \lim \limits _{n \rightarrow \infty} \delta_{n,d}
        = \frac{d^d}{(d+1)^{d+1}}.
\]
\end{prop}

\begin{proof}
This follows from Lemma \ref{lem:dineq} by letting $n \rightarrow \infty$.
\end{proof}

\begin{lemma}
\label{lem:dtineq}
Denote $t= n-d \ge 1$ and assume that $2t \le n+1$.
Then the following estimates hold.
\begin{equation}
\label{estd}
    \frac{1}{n} + \epsilon_{n,d}
        \le \delta_{n,d}
        \le \frac{1}{n} + \epsilon_{n,d} + \frac{(t-1)n}{(n-t+1)^2}
\end{equation}
where
\[
    \epsilon_{n,d} =
        \frac{1}{\binom{n}{t-1}}
            \sum_{i= \left\lceil \frac{(d+1)(t-1)}{t} \right\rceil}^{d-1}
            \frac{\binom{i+t-1}{t-1} \left(\frac{(i+1)t}{d+1}-t+1 \right)}{i+1}.
\]
%\medskip
\end{lemma}

\begin{proof}
We apply \eqref{ii} again and use the fact that $\binom{n}{d+1}=\binom{n}{t-1}$. This yields
\begin{align*}
    \delta_{n,d}
        &= \frac{(n-k)(k-d+1)(k-d+2)\cdots k}{(n-d)(n-d+1)\cdots n} \\
        & \qquad + \frac{1}{\binom{n}{t-1}} \sum_{i=1}^{d-1}
            \frac{\binom{k_i}{i}}{(i+1)}
            \left( \frac{(i+1)(n+1)}{d+1}-k_i-1 \right).
\end{align*}
Because $2t\le n+1$, we have
$n-1\le \frac{d(n+1)}{d+1} = \frac{n^2-nt+n-t}{n-t+1} <n$
and hence $k=n-1$.
Thus
$$
\frac{(n-k)(k-d+1)(k-d+2)\cdots k}{(n-d)(n-d+1)\cdots n}=\frac{1}{n}.
$$
and
\begin{equation}
\label{nii}
    \delta_{n,d} = \frac{1}{n}
      +\frac{1}{\binom{n}{t-1}} \sum_{i=1}^{d-1} \frac{\binom{k_i}{i}}{(i+1)}
        \left( \frac{(i+1)(n+1)}{d+1}-k_i-1 \right).
\end{equation}

We know that
$k_i=\left\lfloor \frac{i(n+1)}{d+1} \right\rfloor = i+ \left\lfloor \frac{it}{d+1} \right\rfloor.$
Since $1\le i \le d-1$ we have $\frac{it}{d+1}<t$ and $0 \le \left\lfloor \frac{it}{d+1} \right\rfloor \le t-1.$
For each $0 \le j \le t-1$, it holds that
$\left\lfloor \frac{it}{d+1} \right\rfloor =j$
if and only if
$$j(d+1)\le it \leq (j+1)(d+1)-1.$$
or equivalently,
\begin{equation}
\label{equ:jj}
    \left\lceil \frac{j(d+1)}{t} \right\rceil
        \le i \le \left\lfloor \frac{(j+1)(d+1)-1}{t} \right\rfloor.
\end{equation}
In particular, it holds that $k_i = i + j$ if the condition above is satisfied.

We now rewrite formula \eqref{nii} using $t = n-d$ again.
\begin{align*}
    \delta_{n,d} - \frac{1}{n}
        &= \frac{1}{\binom{n}{t-1}} \sum_{i=1}^{d-1} \frac{\binom{k_i}{i}}{i+1} \left(\frac{(i+1)(n+1)}{d+1}-k_i-1\right) \\
        &= \frac{1}{\binom{n}{t-1}} \sum_{j=0}^{t-1}
            \left[
            \sum _{i= \left\lceil \frac{j(d+1)}{t} \right\rceil}^{
            \left\lfloor \frac{(j+1)(d+1)-1}{t} \right\rfloor}
            \frac{\binom{i+j}{j}}{i+1} \left(\frac{(i+1)t}{d+1}-j\right)
            \right] %\\
%       &= \epsilon_{n,d}+ \frac{1}{\binom{n}{t-1}}\sum \limits _{j=0}^{t-2}
%            \left[
%            \sum _{i= \left\lceil \frac{j(d+1)}{t} \right\rceil}^{
%            \left\lfloor \frac{(j+1)(d+1)-1}{t} \right\rfloor}
%            \frac{\binom{i+j}{j}}{i+1} \left(\frac{(i+1)t}{d+1}-j\right)
%            \right].
\end{align*}
which is equivalent to
\begin{equation}
\label{ix}
    \delta_{n,d} = \frac{1}{n} + \epsilon_{n,d}\\
        +\frac{1}{\binom{n}{t-1}}
        \sum_{j=0}^{t-2}
            \left[
            \sum _{i= \left\lceil \frac{j(d+1)}{t} \right\rceil}^{
            \left\lfloor \frac{(j+1)(d+1)-1}{t} \right\rfloor}
            \frac{\binom{i+j}{j}}{i+1} \left(\frac{(i+1)t}{d+1}-j\right)
            \right]
\end{equation}
We note that the last sum above is non-negative and get the first inequality in \eqref{estd}.

Now we prove the second inequality in \eqref{estd} by bounding the last summand in formula \eqref{ix} from above. Denote this summand by $F$.
%In order to prove \eqref{estd}, it is enough to show that $$F \le \frac{(t-1)n}{(n-t+1)^2}.$$
\[
    F = \frac{1}{\binom{n}{t-1}}
        \sum_{j=0}^{t-2}
            \left[
            \sum _{i= \left\lceil \frac{j(d+1)}{t} \right\rceil}^{
            \left\lfloor \frac{(j+1)(d+1)-1}{t} \right\rfloor}
            \frac{\binom{i+j}{j}}{i+1} \left(\frac{(i+1)t}{d+1}-j\right)
            \right]
\]
We have
$$
    \frac{\binom{i+j}{j}}{i+1} \left(\frac{(i+1)t}{d+1}-j\right)
        \le \frac{t}{d+1}\binom{i+j}{j},
$$
so
$$
    F \leq \frac{1}{\binom{n}{t-1}}
        \sum_{j=0}^{t-2}
            \left[
            \sum _{i= \left\lceil \frac{j(d+1)}{t} \right\rceil}^{
            \left\lfloor \frac{(j+1)(d+1)-1}{t} \right\rfloor}
            \frac{t}{d+1}\binom{i+j}{j}
            \right].
$$

As $i$ and $j$ that appear in the sum above satisfy formula \eqref{equ:jj}, we find that $i+j \le \frac{(t-1)(d+1)-1}{t} + t-2 \le d+t = n$.
And since $j\le t-2 < \frac{n}{2}$, it is clear that
$\binom{n}{j} \le \binom{n}{t-2}$. Thus,
\begin{equation}
\label{estij}
    F \leq \frac{1}{\binom{n}{t-1}}
        \sum_{j=0}^{t-2}
            \left[
            \sum _{i= \left\lceil \frac{j(d+1)}{t} \right\rceil}^{
            \left\lfloor \frac{(j+1)(d+1)-1}{t} \right\rfloor}
            \frac{t}{d+1} \binom{n}{t-2}
            \right].
\end{equation}

Denote
$\phi (j) := \left\lfloor \frac{(j+1)(d+1)-1}{t} \right\rfloor -
    \left\lceil \frac{j(d+1)}{t} \right\rceil +1$.
Then, we have
\begin{equation}
\label{xi}
    \phi(j) \le \frac{(j+1)(d+1)-1}{t}-\frac{j(d+1)}{t}+1 = \frac{n}{t}.
\end{equation}

Using \eqref{estij} and then \eqref{xi} we get
\begin{align*}
\label{xiii}
    F &\leq \frac{1}{\binom{n}{t-1}}
            \sum_{j=0}^{t-2} \phi(j) \frac{t}{d+1}\binom{n}{t-2}\\
        &\le \frac{1}{\binom{n}{t-1}}
            \sum_{j=0}^{t-2} \frac{n}{d+1}\binom{n}{t-2}
        = \frac{n(t-1) \binom{n}{t-2}}{(d+1) \binom{n}{t-1}}.
\end{align*}
It is easy to see that this implies the second inequality in \eqref{estd}.
\end{proof}

\begin{prop}
For all $t\ge 1$, it holds that
\[
    \lim_{n \rightarrow \infty} \delta_{n,n-t} = \frac{1}{t}.
\]
\end{prop}

\begin{proof}
Direct computation shows that
\[
    \lim_{n\rightarrow \infty} \epsilon_{n,n-t} = \frac{1}{t}.
\]
The result follows then from Lemma \ref{lem:dtineq} by letting $n \rightarrow \infty$.
\end{proof}

\section*{Acknowledgments}
The authors wish to thank Professors Mats Boij  and Ralf Fr\"oberg for the proposal of the problem and valuable conversations concerning this paper. We also wish to thanks Professor Alfio Ragusa and all other organizers of PRAGMATIC 2011 for the opportunity to participate and for the pleasant atmosphere they provided during the summer school.

%\section{Remaining questions}

%\bibliographystyle{plain}
%\bibliography{References}

\end{document}